\documentclass[reqno,twoside,11pt]{amsart}

\usepackage{latexsym,amssymb,amsmath,amsfonts,epsfig,textcomp,mathdots,times}
\usepackage{pstricks,pst-coil} 

\def\sqr#1#2{{\vcenter{\hrule height.#2pt
        \hbox{\vrule width.#2pt height#1pt \kern#1pt
                \vrule width.#2pt}
        \hrule height.#2pt}}}

\def\tratto{\mbox{\rule{2mm}{.2mm}$\;\!$}}

\newtheorem{Theorem}{Theorem}[section]
\newtheorem{Discussion}[Theorem]{Discussion}

\newtheorem{Corollary}[Theorem]{Corollary}
\newtheorem{Proposition}[Theorem]{Proposition}

\newtheorem{Notation and Discussion}[Theorem]{Notation and Discussion}

\newtheorem{Remark}[Theorem]{Remark}

\def\m{\mathfrak m}
\def\a{\mathfrak a}

\def\w{\widetilde}

\def\tha{ \, -\!\!\!\twoheadrightarrow}
\def\bs{\bigskip}
\def\ms{\medskip}
\def\ss{\smallskip}

\def\F{ F}
\def\G{ G}

\def\lto{\longrightarrow}
\begin{document}

\baselineskip=16pt

\title[Iterated socles and integral dependence in regular rings]
{\Large\bf Iterated socles and integral dependence \\ in regular rings}

\author[A. Corso, 
S. Goto, 
C. Huneke, C. Polini, and B. Ulrich]
{Alberto Corso, 
Shiro Goto, 
Craig Huneke, Claudia Polini, \and Bernd Ulrich}
\address{Department of Mathematics, University of Kentucky,
Lexington, KY 40506, USA} \email{alberto.corso@uky.edu}

\address{Department of Mathematics, School of Science and Technology, Meiji University, 1-1-1 Higashi-mita, 
\newline \mbox{\,\quad} Tama-ku, Kawasaki 214-8571, Japan}
\email{goto@math.meiji.ac.jp}

\address{Department of Mathematics, University of Virginia,
Charlottesville, VA 22904, USA} \email{huneke@virginia.edu}

\address{Department of Mathematics, University of Notre Dame,
Notre Dame, IN 46556, USA} \email{cpolini@nd.edu}

\address{Department of Mathematics, Purdue University,
West Lafayette, IN 47907, USA} \email{ulrich@math.purdue.edu}



\thanks{The third author was partially supported by NSF grant DMS-1259142.
The fourth author was partially supported by NSF grant DMS-1202685 and NSA grant H98230-12-1-0242.
The last author was partially supported by NSF grant DMS-1205002 and as a Simons Fellow.}

\begin{abstract}

Let $R$ be a formal power series ring over a field, with maximal ideal $\m$, and let $I$ be an ideal of 
$R$ such that $R/I$ is Artinian. We study the iterated socles of $I$, that is the ideals which are defined 
as the largest ideal $J$ with $J\m^s\subset I$ for a fixed positive integer $s$.
We are interested in these ideals in connection with the notion of integral dependence of ideals. 
In this article we show that the iterated socles
are integral over $I$, with reduction number one,  provided $s \leq \text{o}(I_1(\varphi_d))-1$, where $\text{o}(I_1(\varphi_d))$ 
is the order of the ideal of entries of the last map in a minimal free $R$-resolution of $R/I$. 
In characteristic zero, we also provide formulas for the generators of iterated socles whenever 
$s\leq \text{o}(I_1(\varphi_d))$. This result  generalizes previous 
work of Herzog, who gave formulas for the socle generators of any ${\mathfrak m}$-primary homogeneous ideal $I$ in terms 
of Jacobian determinants of the entries of the matrices in a minimal 
homogeneous free $R$-resolution of $R/I$.
Applications are given to iterated socles of determinantal ideals with generic height. 
In particular, we give surprisingly simple formulas for iterated socles of
height two ideals in a power series ring in two variables. 
These generators  are suitable determinants obtained from the Hilbert-Burch matrix.
\end{abstract}

\maketitle

\section{Introduction}

The {\it socle} of an Artinian  local ring $(A, {\mathfrak m}_A)$  is the annihilator of the maximal ideal, 
$0 :_A {\mathfrak m}_A$. It is the unique largest submodule which has the structure of a module over 
the residue field $k=A/{\mathfrak m}_A$. When $A$
is equicharacteristic, we can write $A = R/I$, where $R=k[\![x_1, \ldots, x_d]\!]$ is a formal powers series ring over a field $k$.
We  pull back the socle to the  ideal $I :_R {\mathfrak m}_R$ in $R$, and by abuse of language call this ideal the socle of $I$. 
Socle generators of modules are as important as the minimal generators of the module, to which they are (in some sense) 
dual, but, in general,  they are much harder to find.

\medskip

The computation of the socle is well understood in the complete intersection case: 
If $I$ is generated by $d$ powers series $f_1, \ldots, f_d$ contained 
in the maximal ideal ${\mathfrak m}_R=(x_1, \ldots, x_d)$, standard linkage theory gives that  
$I :_R {\mathfrak m}_R = (I, \det C)$, 
where $C$ is a square transition matrix that writes the $f$'s in terms of the $x$'s.
Moreover, if the characteristic of the field $k$ is 0 and the $f$'s are homogeneous polynomials, 
one can take as $C$ the Jacobian matrix of the $f$'s, by Euler's formula. 
Hence the socle is generated by $I$ together with the determinant of the 
Jacobian matrix, $I :_R {\mathfrak m}_R = (I, |\partial f_i/\partial x_j|)$.
The same  formula holds in the non-graded case, although the result is much less obvious \cite{SS}.
The converse also holds. Namely, the socle of the Artinian algebra $A$ is its Jacobian ideal, 
$0 :_A {\mathfrak m}_A ={\rm Jac}(A)$,  if and only if $A$ is a complete intersection \cite{SS}. 
In other words, only for complete intersections 
one can expect a simple formula only involving derivatives of the generators of the ideal.

In the complete intersection case the generators of the ideal immediately give the matrices in a  free 
$R$-resolution of $R/I$, by means of the Koszul complex. 
Thus a natural generalization is to ask whether, 
in general, one can obtain formulas for the socle using derivatives of the entries of the matrices in the entire free resolution. 
If $I$ is a homogeneous ideal of a power series ring over a field of characteristic $0$, Herzog  \cite{He} 
gave such a formula for the socle generators by means of Jacobian 
determinants of the entries of the matrices in a homogeneous minimal resolution. 
These formulas suffice to deduce, for instance, 
that if $I$ is an ideal of maximal minors then the socle is contained in the ideal of next lower minors. 
This gives rather strong restrictions on where the socle can sit.
Herzog's result also has a recent application in the study of Golod ideals \cite{HH}.


\medskip

{\it Iterated}  (or {\it quasi}) {\it socles} are simply socles of socles. After $s$ iterations one obtains an ideal 
which can be more easily described as $I :_R {\mathfrak m}_R^s$.  If $I$ is a homogeneous ideal in a polynomial ring $S$ and $X$ 
denotes ${\rm Proj}(S/I)$, then the largest ideal defining the same projective variety is the saturation of $I$, which is 
 an iterated socle. 
It is of great interest to understand the difference between $I$ and its saturation. For example, if a projective variety is defined by
the saturated ideal $I$,  a hyperplane section will be defined by the ideal of $I$ together with the linear form corresponding to the
hyperplane, but this ideal will not in general be saturated itself. We can apply our results
to give precise formulas for the saturated ideals defining hyperplane sections of projective varieties of low 
Castelnuovo-Mumford regularity, see Remark \ref{Hyperplane}.  Iterated socles also appear in the study of the scheme of Gorenstein Artin algebras, 
{\bf Gor}$(T)$, having fixed Hilbert function $T$. In particular, Iarrobino makes use  of what he calls Loewy filtrations which 
are defined by means of subquotients of iterated socles \cite{Iarrobino_1986, Iarrobino_1994}. Our work also leads us to define 
what we feel is a powerful concept, {\it distance}, which can be used as an effective replacement of Castelnuovo-Mumford 
regularity in the local case. We explore this notion in \cite{CHPU}, where we apply it to characterize the Cohen-Macaulayness 
and Gorensteinness of associated graded rings and to explore a conjecture on Loewy lengths of Avramov, Buchweitz, Iyengar, 
and Miller \cite{ABIM}.

\medskip

In this paper we study iterated socles from several perspectives. Many natural
questions arise. For instance, Herzog's socle formula is extremely valuable. Are there similar explicit  formulas for iterated socles?
Another problem is related to the study of integral closures of  ideals. Integral closure plays a crucial role, for instance, in the study of 
Hilbert functions, in intersection theory, and in equisingularity theory.
Since it is difficult to compute the integral closure, one would like to find at least a large part of it.  An
obvious place to look for integral elements are iterated socles, which immediately leads to our main motivating question:

\smallskip

\centerline {\it For which values of $s$ is $I :_R {\mathfrak m}_R^s$ contained in the integral closure of $I$?}
\smallskip
\noindent Still another  problem is to relate iterated socles to other ideals  derived from $I$, when more is known 
about the structure of $I$. One example is the result of Herzog mentioned above. If $I$ is determinantal, its socle
lies in the ideal of next  lower size minors. In particular, iterated socles are contained in the ideal of yet lower size minors. Can more be
said? 

\medskip

We provide almost complete answers to these questions.

\medskip

One cannot expect a positive answer to our main motivating question if $s$ is too large. One obstruction for being in the integral 
closure  arises from the order of an ideal. Recall that the {\it order} of an ideal $I$ in a Noetherian local ring 
$(R, {\mathfrak m})$  is defined as
$\text{o}(I) = \sup \{ t \,|\, I \subset {\mathfrak m}^t \}$. If $R$ is regular, the powers of the maximal 
ideal are integrally closed, hence  $\text{o}(I)=\text{o}(\overline{I})$, where $\overline{I}$ denotes the integral closure of $I$. 
This means that passing  to the integral closure of an ideal cannot  lower the order, at least when $R$ is regular. 

There are several past results dealing with our main motivating question. The first result is 
due to Burch \cite{B}.  In the same paper where she proves the Hilbert-Burch theorem, she also shows that if 
$R$ is not regular and $I$ has finite projective dimension then the entire socle lies in the integral closure of $I$. 
Stronger results have been proved for complete intersection ideals, see for instance  \cite{CPV, CP1, CP2, PU, W, WY}. 
The result of Wang \cite{W} says that if $(R,{\mathfrak m})$ is a regular local ring of dimension $d\geq 2$ and $I$ 
is a complete intersection then $(I : {\mathfrak m}^s)^2=I(I : {\mathfrak m}^s)$
provided $s \leq \text{o}(I)-1$.  In other words, the iterated socle is not only integral over the ideal but also 
the reduction number is at most one.
Little is known for socles, or iterated socles, of ideals that are not complete intersections. 
The integral dependence (with reduction number one) of the socle of a Gorenstein ideal  contained in 
the square of the maximal ideal has been proved in  \cite{CHV}.  A connection between iterated socles and adjoints of ideals 
was addressed by Lipman \cite{L}.
 
\medskip

We now explain our results in more detail. In Section 2  we  unify and generalize  
these previous results about integral dependence with reduction number one. Most notably, we are able to eliminate 
the assumption that the ideal be a complete intersection.  The main result of Section 2 deals with iterated socles 
of any ideal $I$ in a power series ring $R$ in $d\geq 2$ variables over a field; we prove that the iterated socle is 
integral over $I$ with reduction number at most one as long as
$s\leq\text{o}(I_1(\varphi_d))-1$, where $(F_{\bullet}, \varphi_{\bullet})$ is a minimal free $R$-resolution of $R/I$ 
(see Theorem~\ref{main_reduction_1}). Notice that if $R/I$ is Gorenstein then  $I_1(\varphi_d)=I$ 
by the self-duality of the resolution, and the above inequality for $s$ simply becomes the condition $s \leq \text{o}(I)-1$\,
required by Wang and other authors. Thus our result  recovers \cite{CHV,  CP1, CP2, CPV, PU, W} in the regular case 
and provides, at the same time, a vast generalization.  
A somewhat surprising feature of the proof is that the precise knowledge of the module structure of the quotient 
$I:_R{\mathfrak m}^{s+1}/I$ suffices to deduce the equality $(I:_R{\mathfrak m}^s)^2=I(I:_R{\mathfrak m}^s)$ back in the ring $R
$. As structural information we use the fact that this quotient is a  direct sum of copies of the canonical module 
$\omega_{R/{\mathfrak m}^{s+1}}$, which in turn embeds into the module of  polynomials in the inverse variables 
in the sense of Macaulay's inverse systems  (see Corollary \ref{canonical}).

\medskip

In Section~3 we generalize Herzog's formula from socles to iterated socles of not necessarily homogeneous ideals 
in a power series ring $R=k[\![ x_1, \ldots, x_d ]\!]$ over a field of characteristic zero.  More precisely, we provide formulas for 
the generators of iterated socles whenever $s\leq \text{o}(I_1(\varphi_d))$ (see Corollaries \ref{corollary1} and 
\ref{corollary2}, which are special cases of Theorems \ref{closed_formula} and \ref{closed_formula2}).  
Our results cannot be obtained by repeatedly applying Herzog's formulas because this would require knowing  
the resolution of all the intermediate iterated socles. 
Our proof reduces to computing cycles in the tensor complex of a minimal free $R$-resolution of $R/I$ and the 
Koszul complex built on suitable monomial complete intersections. The crucial ingredient is  
an ad hoc modification of the classical de Rham differential on this family of Koszul complexes.  Unlike in Herzog's case, 
our modified de Rham differential is a $k$-linear contracting homotopy which is not a derivation. 
It is simply a connection. Yet, this property suffices for our calculations to go through.

\medskip

In Section~4 we provide applications of the formulas obtained in Section~3 to iterated socles of determinantal ideals 
with generic height. In other words, we obtain strong restrictions on where iterated socles of ideals of  minors of matrices 
can sit (see Theorem~\ref{determinatal}). 
Similar results hold for ideals of minors of symmetric matrices and of Pfaffians. 
In particular, we give surprisingly simple formulas for the generators of  iterated socles of height two ideals in 
a power series ring in two variables. 
These generators  are suitable determinants obtained from the Hilbert-Burch matrix (see Theorem~\ref{height2}).

\medskip

In  a  subsequent article  \cite{CHPU} we consider  iterated socles of ideals in non-regular local rings. 
We obtain substantial improvements in this situation, as we are able to quantify the contribution that comes from 
the non-regularity of the ring.

\medskip

For unexplained terminology and background we refer the reader to \cite{BH}, \cite{SH}, and \cite{V}.

\bs
\bs

\section{Reduction number one}

In this section we  prove our main result on integral dependence, Theorem \ref{main_reduction_1}. 
We begin with an observation that will be used throughout the paper.

\begin{Proposition}\label{basic}
Let $R$ be a Noetherian local ring, $M$ a finite $R$-module, $N=R/J$ with $J$ a perfect $R$-ideal of 
grade $g$, and write
$-^{\vee} = {\rm Ext}_R^g(\tratto, R)$.
\begin{itemize}
\item[(a)]
There are natural isomorphisms \, $0 :_M J \cong {\rm Hom}_R(N,M) \cong {\rm Tor}^R_g(N^{\vee}, M)$.

\item[(b)]
Assume $M$ has a resolution $(F_{\bullet},\varphi_{\bullet})$ of length $g$ by finite free $R$-modules. 
If $I_1(\varphi_g) \subset J$,  then there is a 
natural isomorphism
\[
 {\rm Tor}^R_g(N^{\vee}, M) \cong N^{\vee} \otimes_R F_g\, .  
\]
\end{itemize}
\end{Proposition}
\begin{proof}
We first prove (a). As $N=R/J$, there is a natural isomorphism \, $0 :_M J \cong {\rm Hom}_R(N,M) $. Since
$J$ is perfect of grade $g$, we also have
\[
{\rm Hom}_R(N,M)  \cong  {\rm Tor}_g^R(N^{\vee}, M).
 \]
Indeed, let $G_{\bullet}$ be a resolutions of $N$ of length $g$ by finite free 
$R$-modules. Notice  that   $ G^*_{\bullet}[-g]$ is a  resolution of  $N^{\vee}$ of length $g$ by finite free $R$-modules.
Hence we obtain natural isomorphisms 
\begin{eqnarray*}
{\rm Tor}_g^R(N^{\vee}, M) & \cong & H_g (G_{\bullet}^*[-g] \otimes_R M)\\
& \cong & {\rm Ker} ( G_0^* \otimes _R M \longrightarrow G_1^* \otimes_R M )\\
& \cong &   {\rm Ker} ( {\rm Hom}_R(G_0, M) \longrightarrow {\rm Hom}_R(G_1, M))\\
& \cong & {\rm Hom}_R(N,M) \, .
\end{eqnarray*}

As for part (b), notice that 
\[
{\rm Tor}_g^R(N^{\vee}, M)\  \cong  H_g(N^{\vee} \otimes_R F_{\bullet}) \cong N^{\vee} \otimes_R F_g \, ,
\]
where the last isomorphism holds because $N^{\vee} \otimes_R \varphi_{g+1} =0= N^{\vee} \otimes_R \varphi_g$
by our assumption on $I_1(\varphi_g)$. 
\end{proof}

\bs

The above proposition provides strong structural information about the iterated socle $0 :_M {\mathfrak m}^{s}$; 
it implies, under suitable hypotheses, that this colon is a direct sum of  copies of 
the canonical module of $R/{\mathfrak m}^s$.

\begin{Corollary}\label{canonical} Let $(R, {\mathfrak m})$ be a regular local ring of dimension $d$, $M$ a finite
$R$-module, and $(F_{\bullet}, \varphi_{\bullet})$ a minimal free $R$-resolution of $M$. One has
\[
0 :_M {\mathfrak m}^{s} \cong  \omega_{R/{\mathfrak m}^{s}} \otimes_R F_d, 
\]
for every $ s\le {\rm o}(I_1(\varphi_d))$.
\end{Corollary}
\begin{proof}
We apply Proposition~\ref{basic} with $N=R/{\mathfrak m}^s$.
\end{proof}

\bs

In  Proposition~\ref{reduction1} below we formalize the key step in the proof of Theorem~\ref{main_reduction_1}. 


\begin{Proposition}\label{reduction1}
Let $R$ be a commutative ring, $I \subset K$ ideals, $x, y$ elements of $R$, and $W$ a subset of $R/I$ annihilated by some 
power of $x$ so that $xW=yW$ generates $K/I$. Assume that whenever $x^tyw=0$ in $R/I$ for some $t>0$ and $w \in W$, 
then $x^tw=0$ or $yw=0$ in $R/I$. Then
\[
K^2=IK.
\]
\end{Proposition}
\begin{proof}
Let $U, V$ be preimages in $R$ of $W$ and of $xW=yW$, respectively. We prove that if $v_1, v_2$ are in $V$ and $x^tv_1 \equiv 0 \mod I$
for some $t>0$, then $v_1v_2 \equiv v_1'v_2' \mod IK$ for $v_1', v_2'$ in $V$ with $x^{t-1}v_1' \equiv 0  \ {\rm mod} \,  I$. 
Decreasing induction on $t$ then shows that $v_1v_2 \equiv 0  \ {\rm mod} \,  IK$.

We may assume that $v_1 \not\equiv 0 \ {\rm mod} \, I$, and write $v_1 \equiv y u_1 \ {\rm mod} \,  I$, $v_2 \equiv xu_2 \ {\rm mod} \,  I$ for elements 
$u_1, u_2$ of $U$. Now
\begin{eqnarray*}
v_1 v_2 & \equiv & yu_1v_2 \mod IK \\
& \equiv & yu_1xu_2 \mod IK \\
& = & xu_1y u_2.
\end{eqnarray*}
Since $x^tyu_1 \equiv x^tv_1 \equiv 0\mod I$ and $yu_1 \equiv v_1\not\equiv 0 \mod I$, it follows that $x^t u_1 \equiv 0 \mod I$, 
hence $x^{t-1}(xu_1) \equiv 0 \mod I$. Now set $v_1' =xu_1$ and $v_2' =yu_2$.
\end{proof}

\bs

Theorem~\ref{main_reduction_1} below greatly generalizes the results of  \cite{CHV,  CP1, CP2, CPV, PU, W}  in the case 
of a regular ambient ring.  We consider iterated socles $I:\m^s$ of arbitrary ideals $I$ in an equicharacteristic regular local ring
$R$ of dimension $d \geq 2$,  and we prove
that the iterated socles are integral over $I$ with reduction number at most one as long as 
$s\leq\text{o}(I_1(\varphi_d))-1$, where $(F_{\bullet}, \varphi_{\bullet})$ is a minimal free $R$-resolution of $R/I$. This 
inequality replaces the assumption that $I$ be a complete intersection and $s\leq \text{o}(I)-1$ required in the earlier work,  
and hence appears to be the correct condition to fully understand and generalize  \cite{CHV,  CP1, CP2, CPV, PU, W} . 
The bound $s\leq\text{o}(I_1(\varphi_d))-1$ is sharp as can be seen by taking $I=\m$ and $s=1$.

\begin{Theorem}\label{main_reduction_1}
Let $(R, {\mathfrak m})$ be a regular local ring of dimension $d \geq 2$ containing a field,
$I$ an $R$-ideal, and $(F_{\bullet}, \varphi_{\bullet})$ 
a minimal free $R$-resolution of $R/I$. One has
\[
(I : {\mathfrak m}^s)^2=I(I : {\mathfrak m}^s) 
\]
for every $ s\le {\rm o}(I_1(\varphi_d))-1$.
\end{Theorem}
\begin{proof}
After completing we may assume that  $R=k[\![ x_1, \ldots, x_d ]\!]$ is a power series ring over a field $k$. We wish to apply 
Proposition~\ref{reduction1} with $K=I : {\mathfrak m}^s$. 
Corollary \ref{canonical} gives an isomorphism 
\[
(I : {\mathfrak m}^{s+1})/I \cong  \omega_{R/{\mathfrak m}^{s+1}} \otimes_R F_d, 
\]
which restricts to 
\[
(I : {\mathfrak m}^s)/I \cong    \omega_{R/{\mathfrak m}^{s}} \otimes_R F_d \, .
\]

We recall some standard facts about injective envelopes of the residue field and Macaulay's inverse systems. One has
\[
\omega_{R/{\mathfrak m}^{s+1}} \cong E_{R/{\mathfrak m}^{s+1}}(k) \cong 0 :_{E_R(k)} {\mathfrak m}^{s+1} \subset 
E_R(k) \cong k[x_1^{-1}, \ldots, x_d^{-1}].
\]
The $R$-module structure of the latter is given as follows. We use the identification of $k$-vector spaces  $k[x_1^{-1}, \ldots, x_d^{-1}] =
k[x_1, x_1^{-1}, \ldots, x_d, x_d^{-1}]/N$, where $N$ is the subspace spanned by the monomials not in $k[x_1^{-1}, \ldots, x_d^{-1}]$. 
As $N$ is a $k[x_1, \ldots, x_d]$-submodule, the vector space $k[x_1^{-1}, \ldots, x_d^{-1}]$ becomes a module over 
$k[x_1, \ldots, x_d]$, and then over $R$ since each $x_i$ acts nilpotently.  The $R$-submodule $\omega_{R/{\mathfrak m}^{s+1}} =0 : {\mathfrak m}^{s+1}  \subset k[x_1^{-1}, \ldots, x_d^{-1}]$ 
is generated by the set $M$ of monomials of degree $s$ in the inverse variables $x_1^{-1}, \ldots, x_d^{-1}$.

Notice that $x_1^{s+1}M=0$ and $x_iM$ is the set of monomials of degree $s-1$ in the inverse variables. In particular, $x_1M=x_2M$ 
generates the submodule  $ \omega_{R/{\mathfrak m}^{s}} $. Moreover, if $w =x_1^{-a_1} \cdots x_d^{-a_d} \in M$ and $x_1^tx_2 w =0$, 
then $t >a_1$ or $1>a_2$, in which case $x_1^tw=0$ or $x_2w=0$. Now we may apply Proposition~\ref{reduction1} with $x=x_1$, 
$y=x_2$, and $W=M \otimes_R B \subset  \omega_{R/{\mathfrak m}^{s+1}} \otimes_R F_d$ for $B$ any $R$-basis of $F_d$. 
\end{proof}

\bs

\begin{Corollary}\label{corollary2.5}
Let $(R, {\mathfrak m})$ be a regular local ring of dimension $ \geq 2$ containing a field and $I$ an $R$-ideal. If $R/I$ is Gorenstein, 
then
\[
(I : {\mathfrak m}^s)^2=I(I : {\mathfrak m}^s) 
\]
for every $ s\le {\rm o}(I)-1$.
\end{Corollary}
\begin{proof}
The assertion follows from Theorem~\ref{main_reduction_1} and the symmetry of the resolution of $R/I$.
\end{proof}

\bs

For another example showing that the assumption $s \leq \text{o}(I) -1$ is needed in Corollary~\ref{corollary2.5},
let $(R, {\mathfrak m})$ be a power series ring over a field and $I$ a generic homogeneous  ${\mathfrak m}$-primary 
Gorenstein ideal, in the sense that its dual socle generator is a general form, say of degree $2s-2$. 
One has $\text{o}(I)=s$, see for instance \cite[3.31]{IE}. On the other hand $\text{o}(I \colon 
{\mathfrak m}^s)=s-1$, hence $I \colon {\mathfrak m}^s \not\subset \overline{I}$.

\bs

As pointed out in \cite{CHV,  CP1, CP2, CPV, G, PU, W, WY}, the 
case of a regular ambient ring is the worst  as far as integral dependence is concerned. If the ambient ring is not regular,
we can extend the result about integral dependence with reduction number one 
to the range $s \leq \text{o}(I_1(\varphi_d))$. In \cite{CHPU} we extend the range of integrality considerably further 
even though, as a trade off, we do not obtain reduction number one in general.

\bs

\bs

\section{A formula for iterated socles}

The second goal of this article is to provide closed formulas for the generators of iterated 
socles of any finitely generated module $M$ over a Noetherian local ring $(R, \m)$ of 
equicharacteristic zero. After completing and choosing a Cohen presentation, we may assume that
$R=k[\![x_1, \ldots, x_d]\!]$ is a power series ring in $d$ variables over a field $k$ of characteristic zero. Let 
$(F_{\bullet}, \varphi_{\bullet})$ be a minimal free $R$-resolution of $M$. We will use this
resolution to construct $0:_M \m^s$ in the range $s \leq \text{o}(I_1(\varphi_d))$.

Our result generalizes the work of Herzog \cite{He}, who treated the case where $s=1$ and $M=R/I$ 
for $I$ an ideal generated by homogeneous polynomials.  We stress again that our result does not follow by repeating Herzog's  $s$ times, which would require 
knowing the resolution of the socles at each step.  Our construction instead produces the iterated socle 
in one step from the minimal free resolution of $M$. Our approach resembles that of Herzog, but there are serious obstacles that need to be overcome.
Proposition~\ref{decomposition} below allows us to reduce first to the computation of $0:_M J$, where $J$ is a special monomial complete 
intersection. We then need to consider the Koszul complex of
 this complete intersection and define on it a $k$-linear contracting 
homotopy modeled after the usual de Rham differential, which splits the Koszul differential in positive degrees. 
Unfortunately what we construct is not a derivation, yet it allows our calculations to go through.

\medskip

In the setting of Proposition \ref{decomposition} below one has ${\mathfrak m}^s = \cap (x_1^{a_1}, \ldots, x_d^{a_d})$.
Hence the assertion of the  proposition would follow if one could take the intersection out of the colon as
a sum. This is indeed possible by linkage theory if $M=R/I$ is a Gorenstein ring and  $s \leq {\rm o}(I)$. The content 
of the proposition is that even the weaker assumption $s \leq {\rm o}(I_1(\varphi_d))$ suffices.
We use the notation  $\underline{a}=(a_1, \ldots, a_d)$ for a vector in ${\mathbb Z}^d$ and write
$|\underline{a}| =  \displaystyle\sum_{i=1}^d a_i$.


\begin{Proposition}\label{decomposition}
Let $(R,\m)$ be a regular local ring with a regular system of parameters 
$x_1, \ldots ,x_d$, $M$ a finite $R$-module, 
and $(F_{\bullet}, \varphi_{\bullet})$ a minimal free $R$-resolution of $M$. If $s \leq {\rm o}(I_1(\varphi_d))$, then
\[
0 :_M  {\mathfrak m}^s = \sum_{\substack{|\underline{a}|=s+d-1\\ a_i>0}} 0 :_M (x_1^{a_1}, \ldots, x_d^{a_d}).
\] 
\end{Proposition}
\begin{proof}
One has an irreducible decomposition 
\[
{\mathfrak m}^s = \bigcap_{|\underline{a}|=s+d-1} (x_1^{a_1}, \ldots, x_d^{a_d})\, .
\]
Since any of these ideals $J_{\underline{a}} = (x_1^{a_1}, \ldots, x_d^{a_d})$ contains ${\mathfrak m}^s$, we obtain 
\[
\omega_{R/J_{\underline{a}}} \cong {\rm Hom}_R (R/J_{\underline{a}}, \omega_{R/{\mathfrak m}^s}) \cong 
0 :_{\omega_{R/{\mathfrak m}^s}} J_{\underline{a}} \, .
\]
In particular, $\text{ann}_{{}_R} (0 :_{\omega_{R/{\mathfrak m}^s}} J_{\underline{a}}) = J_{\underline{a}}$. Therefore
\[
\text{ann}_{{}_R} \bigg( \sum_{\underline{a}} 0 :_{\omega_{R/{\mathfrak m}^s}} J_{\underline{a}} \bigg)
=  \bigcap_{\underline{a}} J_{\underline{a}} = {\mathfrak m}^s \, .
\]
In other words, \ $\displaystyle\sum_{\underline{a}} 0 :_{\omega_{R/{\mathfrak m}^s}} J_{\underline{a}}$ \ is a faithful
$R/{\mathfrak m}^s$-submodule of the canonical module $\omega_{R/{\mathfrak m}^s}$. As $\omega_{R/{\mathfrak m}^s}$ cannot have 
a proper faithful $R/{\mathfrak m}^s$-submodule, we conclude that 
\[
\omega_{R/{\mathfrak m}^s} = \sum_{\underline{a}} 0 :_{\omega_{R/{\mathfrak m}^s}} J_{\underline{a}} \, .
\]

According to Corollary~\ref{canonical}, the module $E =0 :_M {\mathfrak m}^s$ is isomorphic to 
$\omega_{R/{\mathfrak m}^s} \otimes_R F_d$. Therefore
\[
E = \sum_{\underline{a}} 0 :_E J_{\underline{a}} \, .
\]
Finally, the inclusion ${\mathfrak m}^s \subset J_{\underline{a}}$ gives $0 :_M {\mathfrak m}^s \supset 0 :_M J_{\underline{a}}$,
hence $0 :_E J_{\underline{a}} = 0 :_M J_{\underline{a}}$.
\end{proof}

\bs

In the next discussion we set up the interpretation of cycles in tensor complexes that we will use to obtain explicit formulas for Koszul cycles and, eventually, for iterated socles. 

\ss

\begin{Discussion}\label{staircase} {\rm 

Let $R$ be a Noetherian ring, let $M,N$ be finite $R$-modules, and let $\F_{\bullet}, \G_{\bullet}$  
be resolutions of $M, N$ by finite free $R$-modules with augmentation maps $\pi, \rho$, respectively. The graded $R$-module $\text{Tor}_{\bullet}^R(M,N)$ 
can be identified with these homology modules,
\[ 
H_{\bullet}(\F_{\bullet} \otimes_R N) \cong H_{\bullet}(\F_{\bullet} \otimes_R \G_{\bullet})\cong H_{\bullet}(M \otimes_R \G_{\bullet})\, .
\]
More precisely, the maps


\begin{center}
\begin{pspicture}(-3,-0.2)(6,1.3)
\psset{xunit=.5cm, yunit=.5cm}

\rput(3.5,2.1){$F_{\bullet} \otimes_R G_{\bullet}$} 

\psline[linestyle=solid,linewidth=0.55pt,
]{->}(3,1.6)(0.2,0.4)
\psline[linestyle=solid,linewidth=0.55pt]{->}(3.8,1.6)(6.6,0.4)

\rput(0.5,1.1){${}^{{\rm id}_{\bullet} \otimes \rho}$}

\rput(6.2,1.1){${}^{\pi \otimes {\rm id}_{\bullet}}$}

\rput(0,-0.1){$F_{\bullet} \otimes_R N$}  \rput(7,-0.1){$M \otimes_R G_{\bullet}$} 

\end{pspicture}
\end{center}
induce epimorphisms on the level of cycles and isomorphisms on the level of homology.  

Recall that an element $\alpha=
(\alpha_0,  \ldots, \alpha_t)$ of $[\F_{\bullet} \otimes_R \G_{\bullet}]_t=\displaystyle\bigoplus_{i=0}^t \F_i \otimes \G_{t-i}$ is a cycle in 
$\F_{\bullet} \otimes_R \G_{\bullet}$ if and only if 
\[ (\text{id}_{\bullet} \otimes \partial^{\G_{\bullet}}) (\alpha_i)= (-1)^{i+1}(\partial^{\F_{\bullet}} \otimes \text{id}_{\bullet})(\alpha_{i+1})
\]
for all $i$. To make the isomorphism $H_{\bullet}(\F_{\bullet} \otimes_R N) \stackrel{\cong}\longrightarrow  H_{\bullet}(M \otimes_R \G_{\bullet})$ 
explicit, we take an arbitrary cycle $v \in [Z(F_{\bullet} \otimes_R N) ]_t$. Lift $v$ to an element $\alpha_t \in F_t \otimes G_0$ with 
$(\text{id} \otimes \rho)(\alpha_t)=v$. Now $\alpha_t$ can be extended to a cycle $\alpha=(\alpha_0,  \ldots,, \alpha_t)\in  
[Z(F_{\bullet} \otimes_R G_{\bullet})]_t$. The image $u=(\pi \otimes \text{id})(\alpha_0)$ is in $[Z(M \otimes_R G_{\bullet})]_t$, and 
the class of $u$ is the image of the class of $v$ under the above isomorphism. 


\begin{center}
\begin{pspicture}(-1.6,-0.2)(18,6.7)
\psset{xunit=.5cm, yunit=.5cm}

\rput(20.9,13.5){$\alpha_0$} \rput(25.2,13.5){$u$}

\psline[linestyle=solid,linewidth=0.55pt]{->}(21.6,13.5)(24.7,13.5) \psline[linestyle=solid,linewidth=0.55pt](21.6,13.4)(21.6,13.6)

\rput{270}(20.9,12.75){$\in$} \rput{270}(25.2,12.75){$\in$}

\rput(20.9,12){$F_0 \otimes G_t $} \psline[linestyle=solid,linewidth=0.55pt]{->>}(22.6,12)(23.6,12) \rput(25.2,12){$M \otimes G_t$}

\psline[linestyle=solid,linewidth=0.55pt]{->}(20.9,11.5)(20.9,10.5)

\rput(13.45,10){$\alpha_1 \in$}
\rput(16.28,10){$ F_1 \otimes G_{t-1}$} \psline[linestyle=solid,linewidth=0.55pt]{->}(18.35,10)(19.35,10) \rput(21.3,10){$F_0 \otimes  G_{t-1}$}

\psline[linestyle=solid,linewidth=0.55pt]{->}(15.9,9.5)(15.9,8.5)

\rput(14.7,7.2){$\cdot$} \rput(15.2,7.6){$\cdot$} \rput(15.7,8){$\cdot$}
\rput(13.2,6){$\cdot$} \rput(13.7,6.4){$\cdot$} \rput(14.2,6.8){$\cdot$}

\psline[linestyle=solid,linewidth=0.55pt]{->}(12.9,5.5)(12.9,4.5)



\rput(4.1,4){$\alpha_{t-1} \in$}
\rput(7.4,4){$ F_{t-1} \otimes G_{1}$} \psline[linestyle=solid,linewidth=0.55pt]{->}(9.5,4)(10.5,4) \rput(12.6,4){$F_{t-2} \otimes G_1$}

\psline[linestyle=solid,linewidth=0.55pt]{->}(7.7,3.5)(7.7,2.5)

\rput(0.1,2){$\alpha_t \in$}
\rput(2.65,2){$ F_t \otimes G_0$} \psline[linestyle=solid,linewidth=0.55pt]{->}(4.35,2)(5.35,2) \rput(7.4,2){$ F_{t-1} \otimes G_0$}

\psline[linestyle=solid,linewidth=0.55pt]{->}(-0.25,1.5)(-0.25,0.5) \psline[linestyle=solid,linewidth=0.55pt](-0.35,1.5)(-0.15,1.5)
\psline[linestyle=solid,linewidth=0.55pt]{->>}(2.55,1.5)(2.55,0.5)

\rput(0.1,0){$v \in$}
\rput(2.5,0){$F_t \otimes N$}


\end{pspicture}
\end{center} }
\end{Discussion}

\bs

\bs

In the next discussion we construct a contracting homotopy for certain complexes $G_{\bullet}$ that allows us to 
invert the vertical differentials in the staircase of Discussion~\ref{staircase}. This provides an explicit formula to pass from 
an element $v$ as above to an element $u$. 

\ss
\begin{Discussion} \label{nabla}
{\rm  
Let $A=k[\![ y_1, \ldots, y_n]\!] \supset A'=k[y_1, \ldots, y_n]$, where $k$ is a field and $y_1, \ldots, y_n$ are variables.
We say that an $A$-module $M$ is {\it graded} if $M \cong A\otimes_{A'} M'$ for a graded $A'$-module $M'$ with respect 
to the standard grading of $A'$. 
One defines, in the obvious way, homogeneous maps and homogeneous complexes of 
graded $A$-modules. We call an element $u$ of a graded $A$-module $M=A\otimes_{A'} M'$ 
homogeneous of degree $d$ if $u = 1 \otimes u'$ for a homogeneous element $u'$ of $M'$ of degree $d$. Notice that every 
element $u$ of $M$ can be uniquely written in the form $u = \sum_{i \in {\mathbb Z}} u_i$ with $u_i$ 
homogeneous of degree $i$.

\smallskip

We consider the universally finite derivation  \[d:A \longrightarrow \Omega_k(A)=F=Ae_1 \oplus \ldots  \oplus Ae_n\] 
with $e_i=dy_i$ homogeneous of degree 1. Let $L_{\bullet}=K_{\bullet}(y_1,\ldots,y_n;A)=\displaystyle\bigwedge^{\bullet} F$ be the 
Koszul complex of $y_1, \ldots, y_n$ with differential $\partial_{\bullet}$ mapping $e_i$ to $y_i$. We extend $d$ to a $k$-linear map 
$d_{\bullet}: L_{\bullet} \longrightarrow L_{\bullet}[1]$, homogeneous  with respect to the internal and the homological degree, by 
setting $d(av)=d(a) \wedge v$ for $a \in A$ and $v=e_{\nu_1} \wedge \ldots \wedge e_{\nu_{\ell}}$ a typical basis element in the 
Koszul complex. After restriction to any graded strand $L_{\bullet_m}$ with respect to the internal grading, one has
\begin{equation}\label{e2}
{ \partial_{\bullet} d_{\bullet}+ d_{\bullet} \partial_{\bullet}}_{| L_{\bullet_m}}= m \,  \text{id}_{L_{\bullet_m}} \, ,
\end{equation}
as can be seen from the Euler relation.

If ${\rm char} \, k=0$ we define 
\[
\w{d}_{\bullet}:{L_{\bullet}}_{>0} \longrightarrow {L_{\bullet}}_{>0} [1] 
\] 
by $\w{d}_{\bullet}(\eta)= \displaystyle\sum_{m>0} \displaystyle \frac{1}{m} d_{\bullet}(\eta_m), $ where $\eta_m$ denotes the degree $m$ component of  $\eta$ 
in the internal grading. The equality (\ref{e2}) shows that $\w{d}_{\bullet}$ is a $k$-linear contracting homotopy of ${L_{\bullet}}_{>0}$.

Now let $R=k[\![ x_1, \ldots, x_n ]\!]$ be another power series ring and let $a_1, \ldots, a_n$ be positive integers. We consider the subring 
$A=k[\![y_1, \ldots,y_n]\!]$ where $y_i=x_i^{a_i}$ are homogeneous of degree 1. Write $V=\displaystyle\bigoplus_{0\le \nu_i < a_i} \, k \, x_1^{\nu_1}\cdots x_n^{\nu_n}$. 
One has $R \cong V \otimes_k A$ as $A$-modules. Thus $R$ is a free $A$-module, which we grade by giving the elements of $V$ degree 0.  
With $K_{\bullet}$ denoting the  Koszul complex $K_{\bullet}(y_1, \ldots, y_n; R)$ we obtain isomorphisms of  complexes of graded 
$A$-modules 
\[
K_{\bullet}\cong R \otimes_A L_{\bullet} \cong V \otimes_k L_{\bullet} \, .
\] 
We define $\nabla_{\bullet}: K_{\bullet} \longrightarrow K_{\bullet}[1]$ 
by $\nabla_{\bullet} = V \otimes_k d_{\bullet}$, which is a $k$-linear homogenous map with respect to the internal and the homological 
degree. We notice that $\nabla_0: R \longrightarrow R \otimes_A \Omega_k(A)$ is no longer a derivation but only a connection of $A$-modules, 
which means it satisfies the product rule if one of the factors is in $A$.
Again, if  ${\rm char} \, k=0$ we define 
\[
\w{\nabla}_{\bullet}:{K_{\bullet}}_{>0} \longrightarrow {K_{\bullet}}_{>0} [1]
\]
by $\w{\nabla}_{\bullet}(\eta)= 
\displaystyle\sum_{m>0} \displaystyle\frac{1}{m} \nabla_{\bullet}(\eta_m)$. Alternatively, one has the description 
$\w{\nabla}_{\bullet}=V \otimes_k \w{d_{\bullet}}$. Hence by the discussion above,  $\w{\nabla}_{\bullet}$ is a $k$-linear contracting homotopy of 
${K_{\bullet}}_{>0}$.

We now describe the maps $\nabla_{\bullet}$ and $\w{\nabla}_{\bullet}$ more explicitly.  Let $S[x]$ be a polynomial ring in one 
variable over a commutative ring $S$ and let $a$ be a positive integer. We consider the $S$-linear map
$\frac{d \ }{dx^a}: S[x] \to S[x]$ with
\[
\frac{d \ }{dx^a}(x^b) = \left\lfloor\frac{b}{a} \right\rfloor x^{b-a}  \, .
\]
We write $S[y]$ for the polynomial subring $S[x^a]$ and $U$ for the free $S$-module $U=\displaystyle\bigoplus_{0\le \nu < a} \, S \, x^{\nu}$. 
We have $S[x] \cong U \otimes_S S[y]$ and $\frac{d \ }{dx^a}= U \otimes_S \frac{d}{dy}$. If $R=k[x_1, \ldots, x_n]$ is a polynomial ring in 
several variables and $a_1, \ldots, a_n$ are positive integers as above, we write $\frac{\partial \ \ }{\partial x^{a_i}}$ for  $\frac{d \ \ }{d x^{a_i}}$. 
To describe the maps $\nabla_{\bullet}$ and $\w{\nabla}_{\bullet}$, let $r\in R$ and let $v =e_{i_1} \wedge \ldots \wedge e_{i_{\ell}}$ be 
a basis element in the Koszul complex $K_{\bullet}$ of degree $\ell$. It turns out that 
\[ 
\nabla(rv)=\sum_{i=1}^n \frac{\partial r}{\partial x^{a_i}} \, e_i \wedge v\, ,
\]
and if ${\rm char}\, k=0$ and $rv \in {K_{\bullet}}_{>0}$,
\[
\w{\nabla}(rv) = \sum_{i,m >0} \frac{1}{m+\ell}\   \frac{\partial r_m}{\partial x^{a_i}}\, e_i \wedge v\, .
\]}
\end{Discussion}

\bs

\begin{Theorem}\label{closed_formula}
Let $R=k[\![ x_1, \ldots,x_d]\!]$ be a power series ring in the variables $x_1, \ldots, x_d$ over a field $k$ of characteristic zero and let $M$ be a 
finite $R$-module. For $a_1, \ldots, a_d$ positive integers, let  $K_{\bullet}$ denote the  Koszul complex $K_{\bullet}(x_1^{a_1}, \ldots, x_d^{a_d}; R)$ 
and let $\w{\nabla}_{\bullet}$ be defined as in \/{\rm Discussion~\ref{nabla}}. Consider a minimal free $R$-resolution  $(F_{\bullet}, \varphi_{\bullet})$ of $M$ and let 
$W_{\bullet}$ be  a graded $k$-vector space with $F_i\cong W_i \otimes_k R$.
  Assume that $(x_1^{a_1}, \ldots, x_d^{a_d}) \supset I_1(\varphi_t)$ for some $t$ and let
$w_1, \ldots, w_r$ be a $k$-basis of $W_t$. 

Then the $R$-module of Koszul cycles
${\mathcal Z}_t(x_1^{a_1}, \ldots, x_d^{a_d}; M)$ is minimally generated by the images in \, $M \otimes_R K_t$ \, of the $r$ elements 
\[
[(\mbox{\rm id}_{W_{\bullet}}\otimes_k \w{\nabla}_{\bullet}) \circ( \varphi_{\bullet}\otimes_R \mbox{\rm id}_{K_{\bullet}}) ]^t (w_{\ell}\otimes 1)\, ,
\]
where $1 \leq \ell \leq r$.
\end{Theorem}
\begin{proof}
We apply Discussion~\ref{staircase} with $N=R/ (x_1^{a_1}, \ldots, x_d^{a_d})$ and $G_{\bullet}=K_{\bullet}$. Notice that 
$[Z(M\otimes_R G_{\bullet})]_t={\mathcal Z}_t(x_1^{a_1}, \ldots, x_d^{a_d}; M)$ and $[Z(F_{\bullet} \otimes_R N)]_t=F_t 
\otimes_R R/ (x_1^{a_1}, \ldots, x_d^{a_d})$ because $\varphi_t \otimes R/ (x_1^{a_1}, \ldots, x_d^{a_d})=0$  
by our assumption on $I_1(\varphi_t)$. 
Hence the latter $R$-module is minimally generated by the elements $w_{\ell} \otimes 1$, $1 \le \ell \le r$, and 
minimal generators of the former $R$-module can be obtained from these elements by applying the horizontal 
differential $\varphi_{\bullet}\otimes_R \mbox{\rm id}_{K_{\bullet}}$ and taking preimages 
under the vertical differential  $\mbox{\rm id}_{F_{\bullet}} \otimes_R \partial_{\bullet}^{K_{\bullet}}$ of Discussion~\ref{staircase}. 
Instead of taking preimages under $\mbox{\rm id}_{F_{\bullet}} \otimes_R \partial_{\bullet}^{K_{\bullet}}= \mbox{\rm id}_{W_{\bullet}} \otimes_k \partial_{\bullet}^{K_{\bullet}}$ we may apply the map $\mbox{\rm id}_{W_{\bullet}} \otimes_k \w{\nabla}_{\bullet}$ because 
the boundaries of $K_{\bullet}$ are in the subcomplex ${K_{\bullet}}_{>0}$ and $\w{\nabla}_{\bullet}$ is a contracting homotopy for the latter according to Discussion~\ref{nabla}.
\end{proof}

\begin{Corollary}\label{corollary1}
With the assumptions of \,\text{\rm Theorem~\ref{closed_formula}}, the $R$-module
$0:_M (x_1^{a_1}, \ldots, x_d^{a_d})$ is minimally  generated by the images in $M \cong   M \otimes_R K_d $ of the $r$
elements 
\[
[(\mbox{\rm id}_{W_{\bullet}}\otimes_k \w{\nabla}_{\bullet}) \circ( \varphi_{\bullet}\otimes_R 
\mbox{\rm id}_{K_{\bullet}}) ]^d (w_{\ell}\otimes 1)\, ,
\]
where $1 \leq \ell \leq r={\rm dim}_k W_d$.
\end{Corollary}
\begin{proof}
The identification $M \cong M\otimes_R K_d$ induces an isomorphism  between $0:_M (x_1^{a_1}, \ldots, x_d^{a_d})$ and  
${\mathcal Z}_d(x_1^{a_1}, \ldots, x_d^{a_d}; M)$. The assertion then follows from Theorem~\ref{closed_formula}. 
\end{proof}

\ms

When combined with Proposition~\ref{decomposition}, the corollary above provides an explicit minimal generating
set of the iterated socle $0:_M{\m}^s$ in the range $s \le {\text o}(I_1(\varphi_d))$.

\bs

We now provide a description of the generators of $0:_M (x_1^{a_1}, \ldots, x_d^{a_d})$ and, more generally, of the 
Koszul cycles  in terms of Jacobian determinants, as was done by Herzog \cite[Corollary 2]{He} when $a_1=\ldots=a_d=1$ 
and $M$ is a cyclic graded module. To do so we consider $M$ as a module over the subring $A=k[\![ y_1, \ldots,y_d]\!]\subset R$, 
where $y_i=x_i^{a_i}$ are homogenous of degree one. 
Applying Herzog's method directly would lead to generators of the Koszul cycles as $A$-modules, whereas using Theorem~\ref{closed_formula} 
above we obtain minimal generating sets as $R$-modules.

We use the assumptions and notations of Theorem~\ref{closed_formula}. Fixing bases of $W_i$ one obtains matrix representations 
$(\alpha^i_{\nu \mu})$ of the maps $\varphi_i$ in the resolution $F_{\bullet}$. Let ${\mathcal B} 
\subset {\mathbb Z}_{\geq 0}^d$ be the set of all tuples $\underline{\lambda}=(\lambda_1, \ldots, \lambda_d)$ with 
$\lambda_i < a_i$ and recall that $\underline{x}^{\underline{\lambda}}$, $\underline{\lambda} \in {\mathcal B}$, form an 
$A$-basis of $R$. For $\underline{\lambda}$ and $\underline{\varepsilon}$ in ${\mathcal B}$, we let $\{ \underline{\lambda}-
\underline{\varepsilon} \}$ be the tuple with
\[
\{ \underline{\lambda}-\underline{\varepsilon} \}_{{}_j} = \begin{cases} \lambda_j - \varepsilon_j & \text{if } 
\lambda_j-\varepsilon_j \geq 0 \\
\lambda_j - \varepsilon_j+a_j & \text{otherwise}.
\end{cases}
\]
Notice that $\{ \underline{\lambda}-\underline{\varepsilon}\}$ is again in ${\mathcal B}$. Moreover we define $\{\!\{ \underline{\lambda} - \underline{\varepsilon} \}\!\}$ by
\[
\{ \! \{ \underline{\lambda}-\underline{\varepsilon} \} \! \}_{{}_j} = \begin{cases} 0 & \text{if } 
\lambda_j-\varepsilon_j \geq 0 \\
1 & \text{otherwise}.
\end{cases}
\]
If $\beta$ is an element of $R$ we write 
$\beta = \displaystyle \sum_{\underline{\gamma} \in {\mathcal B}} \beta_{\underline{\gamma}} \underline{x}^{\underline{\gamma}}$, 
where $ \beta_{\underline{\gamma}} \in A$. Multiplication by $\beta$ gives an $A$-endomorphism of $R$, represented by a matrix 
$M_{\beta}$ with respect to the basis $\underline{x}^{\underline{\lambda}}$, $\underline{\lambda} \in {\mathcal B}$. Its
$(\underline{\lambda}, \underline{\varepsilon})$-entry is
\[
\beta_{\{\underline{\lambda} - \underline{\varepsilon}\} } \,  \underline{y}^{\{\!\{ \underline{\lambda} - \underline{\varepsilon} \}\!\}}.
\]
The $R$-resolution $F_{\bullet}$ of $M$ is also an $A$-resolution. As $A$-basis of $F_{\bullet}$ we choose the $k$-basis of 
$W_{\bullet}$ tensored with the basis $\underline{x}^{\underline{\lambda}}, \underline{\lambda} \in {\mathcal B}$.
To obtain a matrix representation $N_i$ of $\varphi_i$ with respect to this $A$-basis, we replace each $\alpha^i_{\nu \mu}$ 
by the matrix $M_{\alpha^i_{\nu \mu}}$. The $( \nu, \underline{\lambda}; \mu, \underline{\varepsilon})$-entry of $N_i$ is
\[
\alpha^i_{\nu \mu, \{  \underline{\lambda} - \underline{\varepsilon} \}} \, \underline{y}^{\{\!\{  
\underline{\lambda} - \underline{\varepsilon} \}\!\}}.
\]
These entries are in $A$, the power series ring in the variables $y_1, \ldots, y_d$, which all have degree $1$. We consider 
the degree $m$ component $\alpha^i_{\nu \mu, \{  \underline{\lambda} - \underline{\varepsilon} \}, m}$ of 
$\alpha^i_{\nu \mu, \{  \underline{\lambda} - \underline{\varepsilon} \}}$ and we notice that
\[
\alpha^i_{\nu \mu, \{  \underline{\lambda} - \underline{\varepsilon} \}, m} \, \underline{y}^{\{\!\{ \underline{\lambda} - 
\underline{\varepsilon} \}\!\}} 
\]
is homogeneous of degree $m+| \{\!\{ \underline{\lambda} - \underline{\varepsilon} \}\!\}|$.

\bs

\begin{Theorem}\label{closed_formula2}
With the assumptions of \,\text{\rm Theorem~\ref{closed_formula}}, the $R$-module ${\mathcal Z}_t(x_1^{a_1}, \ldots, 
x_d^{a_d}; M)$ is minimally generated by the images in $ M \otimes_R K_t $ of the $r$ elements
\[
\large\sum_{\substack{\nu_{t-1}, \ldots, \nu_0 \\ \underline{\lambda}_{t}, \ldots, \underline{\lambda}_0 \, \text{\rm in} \, {\mathcal B}
\\  m_i >0\\ j_1 < \ldots < j_t}} \
\frac{1}{\displaystyle \prod_{h=t}^1 \displaystyle\sum_{i=t}^{h} (m_i+|\{\!\{  \underline{\lambda}_{i-1} - \underline{\lambda}_i \}\!\}|)}
\ 
\left|
\frac{\partial (\alpha^i_{\nu_{i-1} \nu_i , \{ \underline{\lambda}_{i-1} - \underline{\lambda}_i \}, m_i} \, 
\underline{y}^{\{\!\{ \underline{\lambda}_{i-1} -  \underline{\lambda}_i \}\!\}}) }{\partial y_{j_{\mu}}} 
\right|_{\substack{1\leq i \leq t \\
1 \le \mu \leq t}} \underline{x}^{\underline{\lambda}_0} \otimes e_{j_1} \wedge \ldots \wedge e_{j_t}\, ,
\]
where $1 \leq \nu_t \leq r$.
\end{Theorem} 
\begin{proof}
We claim that the displayed elements in the current theorem and in Theorem~\ref{closed_formula} are equal for $\nu_t=\ell$. 
To prove this we make use of Discussion~\ref{nabla} and the notation introduced there.  In particular, let  $A$ be the subring 
$k[\![ y_1, \ldots, y_d]\!]$ of $R$ with $y_i=x_i^{a_i}$, $V$ the $k$-vector space spanned by the monomial basis of $R$ as an $A$-module, $L_{\bullet}$ the 
Koszul complex $K_{\bullet}(y_1, \ldots, y_d; A)$, $d_{\bullet}$ the de Rham differential of $L_{\bullet}$, and $\w{d}_{\bullet}$ 
the contracting homotopy derived from it. Recall that 
$K_{\bullet} \cong V \otimes_k L_{\bullet}$  and  $\w{\nabla}_{\bullet}=V \otimes_k  \w{d_{\bullet}}$. It follows that 
\[
[(\mbox{\rm id}_{W_{\bullet}}\otimes_k \w{\nabla}_{\bullet}) \circ( \varphi_{\bullet}\otimes_R \mbox{\rm id}_{K_{\bullet}}) ]^t (w_{\nu_t}\otimes 1)= 
[(\mbox{\rm id}_{(W_{\bullet} \otimes_k V) }  \otimes_k \w{d}_{\bullet}) \circ( \varphi_{\bullet}\otimes_A \mbox{\rm id}_{L_{\bullet}}) ]^t (w_{\nu_t}\otimes 1)\, .
\]
Now, to compute the element on the righthand  side of the equation we may consider $(F_{\bullet}, \varphi_{\bullet})$ as an $A$-resolution 
of $M$ with $F_i=W \otimes_k V \otimes_k A$. This element coincides with the one in the current theorem because $d_{\bullet}$, unlike 
$\nabla_{\bullet}$, is a derivation and the operators $\frac{\partial}{\partial y_1}, \ldots,  \frac{\partial}{\partial y_d}$ commute with each other. 
\end{proof}

\bs

\begin{Corollary}\label{corollary2}
With the assumptions of \,\text{\rm Theorem~\ref{closed_formula}}, the $R$-module $0 :_M (x_1^{a_1}, \ldots, x_d^{a_d})$ 
is minimally generated by the images in $ M $  of the $r$ elements
\[
\sum_{\substack{\nu_{d-1}, \ldots, \nu_0 \\ \underline{\lambda}_{d}, \ldots, \underline{\lambda}_0 \, \text{\rm in} \, {\mathcal B}
\\  m_i >0}} \
\frac{1}{\displaystyle \prod_{h=d}^1 \displaystyle\sum_{i=d}^{h}(m_i+|\{\!\{  \underline{\lambda}_{i-1} - \underline{\lambda}_i \}\!\}|)}
\ 
\left|
\frac{\partial (\alpha^i_{\nu_{i-1} \nu_i , \{ \underline{\lambda}_{i-1} - \underline{\lambda}_i \}, m_i} \, 
\underline{y}^{\{\!\{ \underline{\lambda}_{i-1} -  \underline{\lambda}_i \}\!\}}) }{\partial y_j}
\right|_{\substack{1\leq i \leq d \\
1 \leq j \leq d}} \underline{x}^{\underline{\lambda}_0}\, ,
\]
where $1 \leq \nu_d \leq r = \text{\rm rank}\, W_d$.
\end{Corollary}

\ms

\begin{Remark}\label{Hyperplane}{\rm 
Let $R= k[x_1,\ldots, x_d]$ be a polynomial ring over an algebraically closed field $k$ of characteristic zero, and write $\m$ for the
homogeneous maximal ideal. Let $I$ be a homogeneous prime ideal, and let $x\in R_1$ be a general linear form. Fix a minimal 
homogeneous free resolution $(F_{\bullet},\varphi_{\bullet})$ of $R/I$ over $R$. If $I_1(\varphi_{d-1})\subset \m^{{\rm reg}(A)-1}$, then 
Proposition~\ref{decomposition} and
Corollaries~\ref{corollary1} and~\ref{corollary2} can be used to give explicit generators
of the saturated ideal $(I,x):\m^{\infty}$ of the hyperplane section $V(I)\cap V(x)$ in ${\mathbb P}^{d-1}_k$. The point is that this saturation is of the form
$(I,x):\m^s$ for an $s$ that is small enough to apply the full strength of our results. 

To see this, we write $(R',\m') = (R/Rx,\m/Rx)$, $A = R/I$ and
$A'= A/Ax.$ Tensoring $F_{\bullet}$ with $R'$ one obtains a  minimal 
homogeneous free resolution $(F'_{\bullet},\varphi'_{\bullet})$ of $A'$ over $R'.$
Let $\overline{A}$ denote the integral closure of $A$. Notice that $\overline{A}/A$ is a graded $A$-module concentrated in positive degrees,
because $A$ is a positively graded domain over an algebraically closed field. One sees that $H^1_{\m}(A)\cong H^0_{\m}(\overline{A}/A).$ On the
other hand, $H^0_{\m}(A')$ embeds into $H^1_{\m}(A)(-1).$ We deduce that $H^0_{\m}(A')$ is concentrated in degrees $i$, where $2\leq i\leq {\rm reg}(A')={\rm reg}(A)$.
Therefore $\m^{{\rm reg}(A)-1}H^0_{\m}(A') = 0,$ which gives $0:_{A'}(\m')^{\infty} = 0:_{A'}(\m')^{{\rm reg}(A)-1}.$  
On the other hand, ${{\rm reg}(A)-1} \leq \text{o}(I_1(\varphi'_{d-1}))$ by our assumption. Thus 
Proposition~\ref{decomposition} and Corollaries~\ref{corollary1} and~\ref{corollary2} can be
applied to yield the generators of the saturation $0:_{A'}({\m'})^{\infty}.$
 
Such examples occur ``in nature", for example for projective varieties of almost minimal degree which are not Cohen-Macaulay \cite{LeePark}.}
\end{Remark}

\ms

There is a related construction based on differentiating matrices in free resolutions that has been used in the definition of Atiyah 
classes and characteristic classes of modules, see for instance \cite{ALJ}. With $R$, $M$, $(F_{\bullet}, \varphi_{\bullet})$, and 
$W_{\bullet}$ as in Theorem~\ref{closed_formula},  apply the universally finite derivation to the entries of each $\varphi_i$ to obtain an $R$-linear map
\[ 
F_i \otimes_R \bigwedge^{d-i} \Omega_{R/k} \longrightarrow F_{i-1} \otimes_R \bigwedge^{d-i+1} \Omega_{R/k} \, .
\] 
Composing these maps and projecting $F_0$ onto $M$ yields an $R$-linear map 
\[ \Psi: F_d \longrightarrow M \otimes_R \bigwedge^{d} \Omega_{R/k} \, , 
\]
whose class in ${\rm Ext}_R^d(M, M \otimes_R \bigwedge^d \Omega_{R/k})$ only depends on $M$, see \cite[2.3.2]{ALJ}. After identifying $M \otimes_R \bigwedge^{d} \Omega_{R/k}$ with $M$, we consider the image of $\Psi$ as a submodule of $M$. It is natural to try to relate this submodule to the socle of $M$. Indeed, ${\rm im} \Psi = 0:_M \m$ if $M=R/I$ is a complete intersection and $F_{\bullet}$ is the Koszul complex with its natural bases. This is not true in general however. For instance, let $R= k[\![ x, y]\!]$ be a power series ring in the variables $x,y$ over the field $k$ of characteristic zero and consider the free resolution 
\[F_{\bullet}: 0 \longrightarrow R^2 \stackrel{\varphi_2}{\longrightarrow} R^3 \stackrel{\varphi_1}{\longrightarrow} R 
\]
with 
\[
\varphi_2= \left[
\begin{array}{cc}
x^2 & 0 \\
y^2 & x^4\\
0 & x^2+y^3
\end{array} \right]
\qquad \mbox{and}\qquad
\varphi_1 = \left[
\begin{array}{ccc}
x^2y^2+y^5 & -x^4-x^2y^3 & \, x^6
\end{array} \right] \, .
\]

In this case ${\rm im} \Psi $ is generated by the images in $M=R/I=H_0(F_{\bullet})$ of the two elements $x^5y^2$ and $7xy^4+6x^3y$. 
The first element is in $I:\m$, but the second is not. In fact, the second element is not even integral over $I$, whereas $I:\m \subset \overline{I}$ according to Theorem~\ref{main_reduction_1} for instance.
%
To see that $7xy^4+6x^3y$ is not integral over $I$, we give a grading to $R$ by assigning to $x$ degree 3 and to $y$ degree 2. Now $7xy^4+6x^3y$ is homogeneous of degree 11 and $I$ is generated by the homogenous element $x^2y^2+y^5$ of degree 10 and two other homogeneous elements of degrees 12 and 18. Thus if $7xy^4+6x^3y$ were integral over $I$, it would be integral over the principal ideal generated by $x^2y^2+y^5$, and hence would be contained in this ideal, which is not the case.

\bs

\section{Applications to determinantal ideals}

The formulas of  Corollaries~\ref{corollary1} and \ref{corollary2} are somewhat  daunting.  
Nevertheless, they suffice to give strong restrictions on where iterated socles of ideals of  
minors of matrices can sit. Such restrictions also hold for ideals of minors of symmetric matrices 
and of Pfaffians. In particular, this  applies to any height two ideal in a power series ring 
in two variables. In this case the formulas for iterated socles become very simple. 
The generators can be expressed in terms of determinants of the original presentation matrix. 
All these issues will be discussed in the present section.

\smallskip
\begin{Theorem}\label{determinatal} 
Let $(R,\m)$ be a regular local ring containing a field of characteristic zero, let $1\le n \le \ell \le m $ be integers, let $I=I_n(\phi)$ be the ideal generated by 
the $n \times n $ minors of  an $\ell \times m$ matrix $\phi$ with entries in $\m^s$, and assume that $I$ 
has height at least $(\ell-n+1)(m-n+1)$. One has 
\[
I : \m^s \subset I_{n-1}(\phi) \, .
\]
\end{Theorem}
\begin{proof}
Recall that $I$ is a perfect ideal of grade $(\ell-n+1)(m-n+1)$ according to 
\cite{HE}. 
We may assume that $I$ is $\m$-primary and $s\ge 1$. After completing we may 
further suppose that $R=k[\![x_1, \ldots, x_d]\!]$ is a power series ring  in the variables $x_1, \ldots, x_d$  over a field $k$ 
of characteristic zero. 

Now let $Y=(y_{ij})$ be an $\ell \times m$ matrix of variables over $R$, write $S=R[\![\{y_{ij}\}]\!]$, and let $J=I_n(Y)$ be the 
$S$-ideal generated by the $n\times n$ minors of $Y$. We consider a minimal free $S$-resolution $F_{\bullet}$ of $S/J$. 
Since $J$ is extended from an ideal in the ring $k[\![\{y_{ij}\}]\!]$, all matrices of $F_{\bullet}$ have  entries in the $S$-ideal 
$I_1(Y)$ generated by the entries of $Y$.  On the other hand, $J$ is perfect of grade $(\ell-n+1)(m-n+1)$. Hence, regarding $R$ as an $S$-module via the identification $R\cong S/I_1(Y-\phi)$, we see that $F_{\bullet} \otimes_S R$ 
is a minimal free $R$-resolution of $R/I$ and the entries of $Y-\phi$ form a regular sequence on $S/J$. Notice that all matrices 
of the resolution $F_{\bullet} \otimes_S R$ have  entries in the $R$-ideal $I_1(\phi) \subset \m^s$. Thus according to 
Proposition~\ref{decomposition} the assertion of the present theorem follows once we have shown that 
$I : (x_1^{a_1}, \ldots, x_d^{a_d}) \subset I_{n-1}(\phi)$ whenever $|\underline{a}|=s+d-1$. The latter amounts to proving that  
the module of Koszul cycles ${\mathcal Z}_d(x_1^{a_1}, \ldots, x_d^{a_d}; R/I)$ is contained in $I_{n-1}(\phi)\cdot K_d(x_1^{a_1}, \ldots, x_d^{a_d};R/I)$. 
This is a consequence of the next, more general result.
\end{proof}

\smallskip
\begin{Proposition} 
In addition to the assumptions of \ \/{\rm Theorem~\ref{determinatal}} suppose that $R=k[\![x_1, \ldots, x_d]\!]$ is  a power series 
ring in the variables $x_1, \ldots, x_d$ over a field $k$ of characteristic zero.
If $a_i$ are positive integers 
with $\sum_{i=1}^d a_i\le s+d-1$, then 
 \[
 {\mathcal Z}_{t}(x_1^{a_1}, \ldots, x_d^{a_d}; R/I) \subset I_{n-1}(\phi)\cdot K_{t}(x_1^{a_1}, \ldots, x_d^{a_d};R/I)\, 
 \] 
 for every $t >0$. 
\end{Proposition}
\begin{proof}
We adopt the notation introduced in the previous proof. Notice that $\underline{x}^{\underline a}=x_1^{a_1}, \ldots, x_d^{a_d}$ 
form a regular sequence on $S/J$ and so do the entries $\underline{y} -\underline{\phi}$ of the matrix $Y-\phi$. Thus  the Koszul 
complexes $K_{\bullet}(\underline{x}^{\underline a};S/J)$ and $K_{\bullet}(\underline{y} -\underline{\phi};S/J)$ are acyclic. Since 
the second complex resolves $R/I$, Discussion~\ref{staircase} shows that the natural map 
\[K_{\bullet}(\underline{x}^{\underline a};S/J) \otimes_{S/J} K_{\bullet}(\underline{y} -\underline{\phi};S/J) \lto K_{\bullet}(\underline{x}^{\underline a};S/J) 
\otimes_{S/J} R/I \cong K_{\bullet}(\underline{x}^{\underline a}; R/I) 
\]
induces a surjection at the level of cycles. On the other hand, we have the following isomorphisms of complexes
\[K_{\bullet}(\underline{x}^{\underline a};S/J) \otimes_{S/J} K_{\bullet}(\underline{y} -\underline{\phi};S/J)  \cong K_{\bullet}(\underline{x}^{\underline a}, 
\underline{y} -\underline{\phi};S/J) \cong K_{\bullet}(\underline{x}^{\underline a}, \underline{y} ;S/J) \, ,\]
where the first isomorphism follows from the definition of the Koszul complex. The second isomorphism uses the fact that the sequences 
$\underline{x}^{\underline a}, \underline{y} -\underline{\phi}$ and $\underline{x}^{\underline a}, \underline{y} $ minimally generate the 
same ideal in $S$ because $I_1(\underline{\phi}) \subset \m^s \subset (\underline{x}^{\underline{a}})$. 

We conclude that is suffices to show that 
\[
{\mathcal Z}_{t}(\underline{x}^{\underline a}, \underline{y} ;S/J) \subset I_{n-1}(Y)\cdot K_{t}(\underline{x}^{\underline a}, \underline{y} ;S/J)
\]
for every $t>0$. To this end we apply Theorem~\ref{closed_formula} to the ring $S=k[\![\underline{x}, \underline{y}]\!]$, the $S$-module 
$S/J$ with minimal free resolution $(F_{\bullet}, \varphi_{\bullet})$, and the Koszul complex $K_{\bullet}=K_{\bullet}( \underline{x}^{\underline a}, \underline{y} ;S/J)$. 
As observed in the previous proof, the condition $I_1(\varphi_t) \subset (\underline{y})$ is satisfied. Hence Theorem~\ref{closed_formula} yields the inclusion 
\[ 
{\mathcal Z}_{t}(\underline{x}^{\underline a}, \underline{y} ;S/J) \subset S \, [(\mbox{\rm id}_{W_{0}} \otimes_k \w{\nabla}_{t-1}) \circ
(\varphi_{1} \otimes_S \mbox{\rm id}_{K_{t-1}})] (W_1 \otimes K_{t-1}) \, .
\]
The $S$-module  on the right hand side is contained in  $S \,  \w{\nabla}_{t-1}(I_n(Y)\cdot K_{t-1})$. Finally, notice that $\frac{\partial}{\partial y_{ij}}(I_n(Y)) \subset I_{n-1}(Y)$ 
and that $\frac{\partial}{\partial x_{i}^{a_{i}}}(I_n(Y)) \subset I_{n}(Y)$ because the map $\frac{\partial}{\partial x_{i}^{a_{i}}}$ is $k[\![\underline{y}]\!]$-linear. We conclude that
\[S \, \w{\nabla}_{t-1}(I_n(Y)\cdot K_{t-1}) \subset I_{n-1}(Y)\cdot K_t \, ,
\]
as required. 
\end{proof}


\bs

Theorem~\ref{determinatal}  above is sharp, as can be seen by taking $R$ to be a power series ring over a field, $\phi$
a matrix with linear entries, and $n=\ell$. In this case $I \colon {\mathfrak m}= {\mathfrak m}^n \colon {\mathfrak m} 
= {\mathfrak m}^{n-1} \not\subset I_{n}({\phi})$.

\bs

We now turn to perfect ideals of height two. For this it will be convenient  to collect  some general facts of a homological nature.


\begin{Proposition}\label{basic2}
Let $R$ be a Noetherian local ring, $M$ a finite $R$-module, $N=R/J$ with $J$ a perfect $R$-ideal of 
grade $g$, and write $-^*={\rm Hom}_R( \tratto, R)$, 
$-^{\vee} = {\rm Ext}_R^g(\tratto, R)$.
\begin{itemize}

\item[(a)]
If $M$ is perfect of grade $g$, then there is a natural isomorphism \, ${\rm Hom}_R(N, M)  
\cong {\rm Hom}_R(M^{\vee}, N^{\vee})$ given by $u \mapsto u^{\vee}$.

\item[(b)]
If $M$ is perfect of grade $g$ and $\pi: M^{\vee} \tha N \otimes_R M^{\vee}$ 
is the natural projection, then the map $\pi^{\vee}$ is naturally identified with the inclusion map \,
$0 :_M J \hookrightarrow M$.

\item[(c)]
Assume $M$ has a resolution $(F_{\bullet},\varphi_{\bullet})$ of length $g$ by finite free $R$-modules. 
If $I_1(\varphi_g) \subset J$,  then there are 
natural isomorphisms 
\[
\quad \quad \quad  N \otimes_R M^{\vee} \cong 
N \otimes_R F^*_g \quad \mbox{and}  \quad {\rm Hom}_R(M^{\vee}, N^{\vee}) \cong {\rm Hom}_R(F^*_g,N^{\vee})\, .  
\]
\end{itemize}
\end{Proposition}
\begin{proof}

As $-^{\vee}$ is an additive  contravariant functor, we have a linear map  ${\rm Hom}_R(N, M)  
\lto {\rm Hom}_R(M^{\vee}, N^{\vee})$ sending $u$ to $u^{\vee}$. 
The latter is an isomorphism since $-^{\vee\vee} \cong {\rm id}$ 
on the category of perfect $R$-modules of grade $g$. This proves part (a). 

We prove $($b$)$. Let $p : R  \tha R/J$ be the natural projection and notice that $\pi=p \otimes_R M^{\vee}$. 
We show that $(p \otimes_R M^{\vee})^{\vee} $ can be naturally identified with ${\rm Hom}_R (p, M)$.  
Let $\a \subset {\rm ann} \, M$ be an ideal generated by an $R$-regular sequence of length $g$. 
There are natural identifications of maps
\begin{eqnarray*}
(p \otimes_R M^{\vee})^{\vee} &=& {\rm Ext}^g_R(p \otimes_R M^{\vee}, R) \\
&=&  {\rm Hom}_R(p \otimes_R M^{\vee}, R/\a)\\
&=&{\rm Hom}_R(p, {\rm Hom}_R(M^{\vee}, R/\a))\\
&=& {\rm Hom}_R(p, {\rm  Ext}^g_R(M^{\vee}, R))\\
&=&{\rm Hom}_R(p,M^{\vee\vee})\\
&=&{\rm Hom}_R(p,M)\, ,
\end{eqnarray*}
where the last equality uses the assumption that $M$ is perfect of grade $g$.
Finally notice that 
\[
{\rm Hom}_R(p,M): {\rm Hom}_R(R/J,M) \lto {\rm Hom}_R(R,M)
\] 
can be  identified with the inclusion  map \ $0 :_M J \hookrightarrow M$. 

As for part (c), notice that $ M^{\vee}\cong {\rm coker}\, \varphi_g^*$. Hence the containment 
 $I_1(\varphi_g) \subset J={\rm ann}(N)$ implies that  $N \otimes_R M^{\vee} \cong N \otimes_R F^*_g$, 
which is the first isomorphism in part (c).
Now Hom-tensor adjointness gives the second isomorphism 
${\rm Hom}_R(M^{\vee}, N^{\vee}) \cong {\rm Hom}_R(F^*_g,N^{\vee})$ because $N=R/J$.  
\end{proof}

\smallskip

\begin{Corollary}\label{ht2} In addition to the assumptions of \ \/{\rm Proposition~\ref{basic2}} assume that $M$ is perfect of grade $g$. Let $(F_{\bullet},\varphi_{\bullet})$  and $G_{\bullet}$ be resolutions of  $M$ and $N$ of length $g$ by finite free $R$-modules. Notice that  $F^*_{\bullet}[-g]$ and  $ G^*_{\bullet}[-g]$ 
are resolutions of $M^{\vee}$ and $N^{\vee}$ by finite free $R$-modules. 

 \begin{itemize} 

\item[(a)] Given a linear map $v: M^{\vee} \longrightarrow N^{\vee}$, lift $v$ to a morphism of complexes   $ \widetilde{v}_{\bullet} : F^*_{\bullet} \longrightarrow G^*_{\bullet} \, , $ dualize to obtain
$\widetilde{v}^*_{\bullet} : G_{\bullet} \longrightarrow F_{\bullet} \, , $ and consider $H_0(\widetilde{v}^*_{\bullet}) : N=R/J \longrightarrow M$. One has 
\[ 0:_M J = \{ H_0(\widetilde{v}^*_{\bullet})(1+J)  \mid v \in  {\rm Hom}_R(M^{\vee}, N^{\vee})  \}\, .
\]

\item[(b)] Assume that $I_1(\varphi_g) \subset J$. Let  
\[ w: M^{\vee}  \tha N \otimes_R F_g^* 
\] be the composition of the epimorphism $\pi$ in   \ \/{\rm Proposition~\ref{basic2}({b})} with  the first isomorphism in \ \/{\rm Proposition~\ref{basic2}({c})} and lift $w$ to a morphism of 
complexes \[
\widetilde{w}_{\bullet}: F^*_{\bullet} \lto G_{\bullet}[g] \otimes_R F_g^*\, .
\] The mapping cone $C(\widetilde{w}^*_{\bullet} )$ is a free $R$-resolution of $M/(0 :_M J)$.
\end{itemize}
\end{Corollary}
\begin{proof} From   \/{\rm Propositions~\ref{basic2}(a)} and \/{\rm ~\ref{basic}(a)} we have isomorphisms 
\[  
{\rm Hom}_R(M^{\vee}, N^{\vee}) \stackrel{\simeq} {\longrightarrow}{\rm Hom}_R(N, M)  \stackrel{\simeq} {\longrightarrow} 0 :_M J \, ,
\]
where the first map sends $v$ to $v^{\vee}=H_0(\widetilde{v}^*_{\bullet})$ and the second map sends $u$ to $u(1+J)$. This proves part (a).


To show part (b), notice that 
\[ \widetilde{w}^*_{\bullet} : G^*_{\bullet}[-g] \otimes_R F_g \longrightarrow F_{\bullet}\, ,
\]
where  $G^*_{\bullet}[-g] \otimes_R F_g$ and  $F_{\bullet}$ are acyclic complexes of finite free $R$-modules. 
Moreover,  $H_0(\widetilde{w}^*_{\bullet})=w^{\vee}$ and  the latter can be identified 
with the inclusion map  $\ 0 :_M J \hookrightarrow M$ according to Proposition~\ref{basic2}(b).
It follows that the mapping cone $C(\widetilde{w}^*_{\bullet} )$ is a free $R$-resolution of $M/(0 :_M J)$.
\end{proof}

\smallskip
\begin{Theorem}\label{height2}
Let $(R,\m)$ be a two-dimensional regular local ring with regular system of parameters $x,y$, and let $I$ be an ${\mathfrak m}$-primary ideal 
minimally presented by a $n \times (n-1)$ matrix $\phi$ with entries in $\m^s$. 

\begin{itemize}
\item[(a)] For $1 \le i \le n-1$ and $1\le a \le s$ write the $i^{th}$ column $\phi_i$ of $\phi$ in the form $x^{s+1-a}\eta+y^{a}\xi$ and let $\Delta_{ia}$ 
be the determinant of the $n \times n$ matrix obtained from $\phi$ by replacing $\phi_i$ with the two columns $\eta$ and $\xi$. One has 
\[ 
I : \m^s =I + (\Delta_{ia} \mid 1\le i \le n-1, 1\le a\le s).
\]
\item[(b)]  The ideal $I:\m^s$ is minimally presented by the $[(n-1)(s+1)+1] \times (n-1)(s+1)$ matrix 
\[
\psi = \left[
\begin{array}{ccccc}
&&&& \\
&& B && \\ 
&&&&\\ \hline
&&&&\\
&& \chi && \\
&&&&
\end{array}
\right].
\]
The $n \times (n-1)(s+1)$ matrix $B$ is obtained from $\phi$ by replacing each column $\phi_i$  with the $s+1$ columns $\phi_{i0}, \ldots, \phi_{is}$ 
defined by the equation $\phi_i= \sum_{j=0}^s x^{s-j}y^{j}\phi_{ij}$. The $(n-1)s \times (n-1)(s+1)$ matrix $\chi$ is the direct sum of  $n-1$ 
copies of the $s \times (s+1)$ matrix
\[ 
\left[ \begin{array}{rrrrr}
-y & x & & &  \\
& -y & x & &  \\
& & \ddots& \ddots &  \\
& & & -y & x\\
\end{array} \right].
\]
\end{itemize}
\end{Theorem}
\begin{proof}
We prove part (a). The $R$-module $R/I$ has a minimal free $R$-resolution $(F_{\bullet}, \varphi_{\bullet})$ of length 2. After a choice of bases 
we may assume that $\varphi_2=\phi$. As $I_1(\varphi_2) \subset \m^s$, Proposition~\ref{decomposition} shows that 
\[
I :  {\mathfrak m}^s = \sum_{1\le a \le s} I : (x^{s+1-a}, y^{a})\, .
\]
We claim that $I : (x^{s+1-a}, y^{a})= I + (\Delta_{ia} \mid 1\le i \le n-1)$ for every $1\le a\le s$. 

For this we wish to apply Corollary~\ref{ht2}(a) with $M=R/I$, $J=(x^{s+1-a}, y^{a})$, and $G_{\bullet}=K_{\bullet}$ the Koszul complex of 
$-y^{a},x^{s+1-a}$ with its natural bases. The second isomorphism of Proposition~\ref{basic2}({c}) gives 
$ {\rm Hom}_R(M^{\vee}, N^{\vee}) \cong {\rm Hom}_R(F^*_2, \omega_{R/J})$.  Let $v\in {\rm Hom}_R(F^*_2, \omega_{R/J})$ be the 
projection $\pi_i: F^*_2 \tha R=K^*_2$ onto the $i^{th}$ component followed by the epimorphism 
$K^*_2 \tha \omega_{R/J}$. We lift $v$ to a morphism of complexes $\widetilde{v}_{\bullet} : F^*_{\bullet} \longrightarrow K^*_{\bullet}$ 
with $\w{v}_{-2}=\pi_i$. The $R$-module $ {\rm Hom}_R(M^{\vee}, N^{\vee}) \cong {\rm Hom}_R(F^*_2, \omega_{R/J})$ is generated 
by the elements $v$ as $i$ varies in the range $1\le i \le n-1$. Hence according to Corollary~\ref{ht2}(a) the ideal $I : (x^{s+1-a}, y^{a})$ is 
generated by $I$ together with the ideals $I_1(\w{v}^*_0)$. Thus it suffices to prove that  $I_1(\w{v}^*_0)$ is generated by $\Delta_{ia}$. 

In the diagram 

\begin{center}
\begin{pspicture}(-3.8,-0.4)(20,2.2)
\psset{xunit=.5cm, yunit=.5cm}

\rput(2,3.75){$0$} \rput(5,3.75){$F_0^*$}    \rput(10,3.75){$F_1^*$}   \rput(15,3.75){$F_2^*$}

\rput(2,0){$0$} \rput(5,0){$K_0^*$}   \rput(10,0){$K_1^*$} \rput(15,0){$K_2^*$}

\psline[linestyle=solid,linewidth=0.55pt]{->}(5,3)(5,0.75)
\psline[linestyle=solid,linewidth=0.55pt]{->}(10,3)(10,0.75)
\psline[linestyle=solid,linewidth=0.55pt]{->}(15,3)(15,0.75)

\psline[linestyle=solid,linewidth=0.55pt]{->}(2.5,0)(4.3,0)
\psline[linestyle=solid,linewidth=0.55pt]{->}(5.7,0)(9.3,0)
\psline[linestyle=solid,linewidth=0.55pt]{->}(10.7,0)(14.3,0)

\psline[linestyle=solid,linewidth=0.55pt]{->}(2.5,3.75)(4.3,3.75)
\psline[linestyle=solid,linewidth=0.55pt]{->}(5.7,3.75)(9.3,3.75)
\psline[linestyle=solid,linewidth=0.55pt]{->}(10.7,3.75)(14.3,3.75)

\rput(7.5,4.2){${}^{\varphi_1^*}$} \rput(12.5,4.2){${}^{\varphi_2^*\,=\,\phi^*}$}

\rput(7.5,0.45){${}^{\partial_1^*}$} \rput(12.5,0.45){${}^{\partial_2^*}$}

\rput(5.6,1.9){${}^{\widetilde{v}_0}$} \rput(10.8,1.9){${}^{\widetilde{v}_{-1}}$} \rput(16.3,1.9){${}^{\widetilde{v}_{-2}= \pi_i}$}

\end{pspicture}
\end{center}

\noindent
we may choose $\w{v}_{-1}=[\ \eta\  | \ \xi \ ]^*$ because $[x^{s+1-a} \ y^a]\cdot [\ \eta\  | \ \xi \ ]^*=\partial_2^* \  \w{v}_{-1}$ is the $i^{th}$ 
row of $\phi^*$, which equals $\pi_i \ \varphi_2^*$. Since $v$ is surjective, the mapping cone $C(\w{v}_{\bullet})$ is acyclic, and 
hence so is  $C(\widetilde{v}_{\bullet} )^*$ or equivalently $C(\widetilde{v}^*_{\bullet} )$. After splitting off a summand, the latter resolution has the form 

\begin{center}
\begin{pspicture}(-3,1.5)(20,2.7)
\psset{xunit=.5cm, yunit=.5cm}

\rput(.5,3.75){$0$} \rput(5.5,3.75){$ (F_2/\pi_i^*(K_2)) \oplus K_1$}    \rput(14.9,3.75){$F_1 \oplus K_0$}   \rput(20.3,3.75){$F_0$\, ,} 

\psline[linestyle=solid,linewidth=0.55pt]{->}(1,3.75)(2,3.75)
\psline[linestyle=solid,linewidth=0.55pt]{->}(8.9,3.75)(13.2,3.75)
\psline[linestyle=solid,linewidth=0.55pt]{->}(16.6,3.75)(19.4,3.75)

\psline[linestyle=solid,linewidth=0.55pt](9.1,4.2)(9.1,5.9)
\psline[linestyle=solid,linewidth=0.55pt](9.1,4.2)(9.25,4.2)
\psline[linestyle=solid,linewidth=0.55pt](9.1,5.9)(9.25,5.9)

\psline[linestyle=solid,linewidth=0.55pt](12.9,4.2)(12.9,5.9)
\psline[linestyle=solid,linewidth=0.55pt](12.75,4.2)(12.9,4.2)
\psline[linestyle=solid,linewidth=0.55pt](12.75,5.9)(12.9,5.9)

\psline[linestyle=solid,linewidth=0.55pt](10.8,4.2)(10.8,5.9)
\psline[linestyle=solid,linewidth=0.55pt](11.8,5)(11.8,5.9)
\psline[linestyle=solid,linewidth=0.55pt](9.25,5)(12.75,5)

\rput(10.1,5.5){${}^{\phi'}$} \rput(11.3,5.4){${}^{\eta}$} \rput(12.3,5.4){${}^{\xi}$}
\rput(10,4.4){${}^{0}$} \rput(11.8,4.4){${}^{-\partial_1}$}

\rput(17.45,4.4){${}^{\varphi_1}$} \rput(18.6,4.5){${}^{\widetilde{v}_0^*}$} 

\psline[linestyle=solid,linewidth=0.55pt](17.95,4.2)(17.95,4.8)

\psline[linestyle=solid,linewidth=0.55pt](16.8,4.2)(16.8,4.8)
\psline[linestyle=solid,linewidth=0.55pt](16.8,4.2)(16.9,4.2)
\psline[linestyle=solid,linewidth=0.55pt](16.8,4.8)(16.9,4.8)

\psline[linestyle=solid,linewidth=0.55pt](19.1,4.2)(19.1,4.8)
\psline[linestyle=solid,linewidth=0.55pt](19.1,4.2)(19,4.2)
\psline[linestyle=solid,linewidth=0.55pt](19.1,4.8)(19,4.8) 

\end{pspicture} 
\end{center} 

\noindent
where $\phi'$ is obtained from $\phi$ by deleting the $i^{th}$ column. Since $I_1(\varphi_1)=I$ has height 2, the Hilbert-Burch 
Theorem now shows that $I_1(\w{v}^*_0)= I_n([\ \phi' \ | \ \eta\  | \ \xi \ ])=R \Delta_{ia}$, as desired.

To prove part (b) we wish to apply Corollary~\ref{ht2}(b) with $M=R/I$, $J=\m^s$, and $G_{\bullet}$ the resolution that, after a 
choice of bases, has differentials
\[ 
\partial_1 = [ \ x^s \ \ x^{s-1}y \ \ \ldots \ \ y^s]
\qquad
\text{and}
\qquad
\partial_2 = \left[
\begin{array}{rrrr}
y  &   & &   \\
-x &y    & &   \\
& -x  & \ddots &  \\
&   & \ddots &  y \\
& &   &     -x\\
\end{array} \right] \, . 
\]
Lift the natural epimorphism 
\[ 
w: M^{\vee}  \tha R/\m^s \otimes_R F^*_2
\]

\smallskip
\noindent
to a morphism of complexes $ \widetilde{w}_{\bullet} : F_{\bullet}^* \lto G_{\bullet}[-2] \otimes_R F_2^*$ so that $\w{w}_{-2}={\rm id}$. In the diagram

\begin{center}
\begin{pspicture}(-3.8,-0.4)(20,2.2)
\psset{xunit=.5cm, yunit=.5cm}

\rput(2,3.75){$0$} \rput(5,3.75){$F_0^*$}    \rput(10,3.75){$F_1^*$}   \rput(15,3.75){$F_2^*$}

\rput(2,0){$0$} \rput(5,0){$G_2\otimes F_2^*$}   \rput(10,0){$G_1 \otimes F_2^*$} \rput(15,0){$G_0 \otimes F_2^*$}

\psline[linestyle=solid,linewidth=0.55pt]{->}(5,3)(5,0.75)
\psline[linestyle=solid,linewidth=0.55pt]{->}(10,3)(10,0.75)
\psline[linestyle=solid,linewidth=0.55pt](15.1,3)(15.1,0.75)
\psline[linestyle=solid,linewidth=0.55pt](14.9,3)(14.9,0.75)

\psline[linestyle=solid,linewidth=0.55pt]{->}(2.5,0)(3.3,0)
\psline[linestyle=solid,linewidth=0.55pt]{->}(6.6,0)(8.4,0)
\psline[linestyle=solid,linewidth=0.55pt]{->}(11.6,0)(13.4,0)

\psline[linestyle=solid,linewidth=0.55pt]{->}(2.5,3.75)(4.3,3.75)
\psline[linestyle=solid,linewidth=0.55pt]{->}(5.7,3.75)(9.3,3.75)
\psline[linestyle=solid,linewidth=0.55pt]{->}(10.7,3.75)(14.3,3.75)

\rput(7.5,4.2){${}^{\varphi_1^*}$} \rput(12.5,4.2){${}^{\varphi_2^*\,=\,\phi^*}$}

\rput(7.5,0.45){${}^{\partial_2\otimes F_2^* }$} \rput(12.5,0.45){${}^{\partial_1 \otimes F_2^*}$}

\rput(5.6,1.9){${}^{\widetilde{w}_0}$} \rput(10.8,1.9){${}^{\widetilde{w}_{-1}}$} \rput(16.4,1.9){${}^{\widetilde{w}_{-2}= {\rm id}}$}

\end{pspicture}
\end{center}

\noindent
we can choose $\w{w}_{-1}=B^*$ by the definition of $B$. Notice that $\partial_2 \otimes F_2^*=\chi^*$. 
The result now follows since $C(\widetilde{w}^*_{\bullet} )$ is a free $R$-resolution of $R/(I:_R\m^s)$ according to 
Corollary~\ref{ht2}(b). 
\end{proof}

\smallskip
\begin{Remark}{\rm In the setting of Theorem~\ref{height2} a minimal free $R$-resolution of $R/(I:\m^s)$ is 

\begin{center}
\begin{pspicture}(-4.1,-0.4)(20,-0.4)
\psset{xunit=.5cm, yunit=.5cm}

\rput(0.5,0){$0$} \rput(4,0.1){$R^{(n-1)(s+1)}$}    \rput(10.4,0.1){$R^{(n-1)(s+1)+1}$}   \rput(15.3,0){$R\ ,$}

\psline[linestyle=solid,linewidth=0.55pt]{->}(1,0)(2,0)
\psline[linestyle=solid,linewidth=0.55pt]{->}(6,0)(8,0)
\psline[linestyle=solid,linewidth=0.55pt]{->}(12.8,0)(14.8,0)

\rput(7,0.35){${}^{\psi_2}$} \rput(13.8,0.35){${}^{\psi_1}$}

\end{pspicture}
\end{center}

\noindent where $\psi_2=\psi$ and the $l^{th}$ entry of $\psi_1$ is the signed $l^{th}$ maximal minor of $\psi$. If $l \le n$ 
the $l^{th}$ maximal minor of $\psi$ is the $l^{th}$ maximal minor of $\phi$; if $l\ge n+1$ write $l=n+(i-1)s+a$ for 
$1\le i \le n-1$, $1\le a \le s$, and  the $l^{th}$ maximal minor of $\psi$ is $\Delta_{ia}$, where  
$\eta :=\displaystyle\sum_{j=0}^{a-1} x^{a-1-j}y^j \phi_{ij}$ and  $\xi:=\displaystyle\sum_{j=a}^{s}x^{s-j}y^{j-a}\phi_{ij}$ 
are used in the definition of  $\Delta_{ia}$. 

For the proof one uses Theorem~\ref{height2}(b), the Hilbert-Burch Theorem, and the following elementary fact about 
determinants that can be shown 
by induction on $r$ and expansion along the last row:

If $\left[
\begin{array}{ccc}
\ \varepsilon & \vline & \mu  \\ \hline
\ 0 & \vline & \delta  \\
\end{array}
\right]$ 
is a square matrix, where $\mu$ has columns $\mu_0, \ldots, \mu_{r}$ and $\delta=\left[ \begin{array}{cccc}
-y & x & &  \\
&  \ddots& \ddots   & \\
 & &  -y & x \\
\end{array} \right]$, then $\ {\rm det} \left[
\begin{array}{ccc}
\ \varepsilon & \vline & \mu  \\ 
 \hline
\ 0 & \vline & \delta  \\
\end{array}
\right]= {\rm det}\, \left[
\ \varepsilon \  \ \vline  \ \displaystyle\sum_{j=0}^{r} x^{r-j}y^j \mu_j\right]$. }
\end{Remark}

\bs

\ms

\noindent{\bf Acknowledgments.} 
Part of this work was done at the Mathematical Sciences Research Institute (MSRI) in Berkeley, where the authors spent  time in connection 
with the 2012-13 thematic year on Commutative Algebra, supported by NSF grant 0932078000. The authors would like to thank MSRI for its partial support and hospitality. 
The authors would also like thank David Eisenbud for helpful discussions on the material of this paper.

 \bs

\end{document}